\documentclass[11pt,twoside]{amsart}
\usepackage{amsmath,amsfonts,amscd,amsthm}
\usepackage{amssymb}
\usepackage[flushmargin]{footmisc}
\usepackage{color}                    % For creating coloured text and background
\usepackage{esvect}

\DeclareFontFamily{U}{mathx}{\hyphenchar\font45}
\DeclareFontShape{U}{mathx}{m}{n}{
      <5> <6> <7> <8> <9> <10>
      <10.95> <12> <14.4> <17.28> <20.74> <24.88>
      mathx10
      }{}
\DeclareSymbolFont{mathx}{U}{mathx}{m}{n}
\DeclareFontSubstitution{U}{mathx}{m}{n}
\DeclareMathAccent{\widecheck}{0}{mathx}{"71}
\DeclareMathAccent{\wideparen}{0}{mathx}{"75}

\newcommand{\dbar}{\ensuremath{\overline\partial}}
\newcommand{\dbarstar}{\ensuremath{\overline\partial^*}}
\newcommand{\C}{\ensuremath{\mathbb{C}}}
\newcommand{\R}{\ensuremath{\mathbb{R}}}

\def\eps{\varepsilon}

\def\omz{\Omega}

\def\ov{\overline}
\makeatletter
\newcommand{\sumprime}{\if@display\sideset{}{'}\sum%
	\else\sum'\fi}
\makeatother

\newtheorem{thm}{Theorem}[section]
\newtheorem{prop}[thm]{Proposition}
\newtheorem{lem}[thm]{Lemma}

\theoremstyle{definition}

\newtheorem{remark}{Remark}
\numberwithin{equation}{section}

\providecommand\ufootnote[1]{{\let\thefootnote\relax\footnote[0]{#1}}}

\newcommand{\ec}{\mathcal E}

\newcommand{\N}{\mathbb{N}}

\newcommand{\ol}{\overline}

\DeclareMathOperator{\supp}{supp}

\DeclareMathOperator{\dist}{dist} 
 \DeclareMathOperator{\Dom}{Dom}
\DeclareMathOperator{\Span}{Span}  

\begin{document}

\title[Kohn-Nirenberg Elliptic Regularization]{Spectral Stability of the $\dbar-$Neumann Laplacian: \\ the Kohn-Nirenberg elliptic regularization }

\author{Siqi Fu, Chunhui Qiu, and Weixia Zhu}

\thanks
{The first author was supported in part by a grant from the National Science Foundation
	(DMS-1500952). The second and third authors were supported in part by a grant from the National Natural Science Foundation of China (Grant No.~11571288).}

\address{S. Fu, Department of Mathematical Sciences,
	Rutgers University, Camden, NJ 08102, USA} \email{sfu@rutgers.edu}

\address{W. Zhu, School of Mathematical Sciences,
   Xiamen University, Xiamen, Fujian 361005, CHN}\email{zhuvixia@stu.xmu.edu.cn}

\begin{abstract} In this paper we study spectral stability of the $\bar\partial$-Neumann Laplacian under the Kohn-Nirenberg elliptic regularization. We obtain quantitative estimates for stability of the spectrum of the $\bar\partial$-Neumann Laplacian when either the operator or the underlying domain is perturbed.	
\bigskip
	
	
	\smallskip
	
\end{abstract}

\maketitle

%\tableofcontents

\section{Introduction}\label{sec:intro}

The $\dbar$-Neumann Laplacian $\Box_q$ on a bounded domain $\Omega$ in $\C^n$ 
is (a constant multiple of) the usual Laplacian acting diagonally on $(p,q)$-forms with the $\dbar$-Neumann boundary condition. It is the archetype of an elliptic operator with non-coercive boundary condition.   Subelliptic estimates for the $\dbar$-Neumann Laplacian on smoothly bounded strongly pseudoconvex domains in $\C^n$ was established by Kohn~\cite{Kohn63} (see \cite{DangeloKohn99} for an exposition on related subjects).  One difficulty in studying non-coercive boundary value problems is to show that {\it a priori} estimates of derivatives imply that these derivatives exist and the same estimates hold without prior regularity assumptions.  Elliptic regularization was introduced by Kohn and Nirenberg~\cite{KohnNirenberg65} to resolve this difficulty.  By adding a positive constant $t$ multiple of an elliptic operator to the $\dbar$-Neumann Laplacian, the $\dbar$-Neumann problem is converted into a coercive elliptic problem for which existence of the derivatives is well known. One then obtains  {\it bona fide} estimates from {\it a priori} ones by taking $t\to 0^+$, provided the desired estimates are uniform in $t$. 

Spectral stability for the classical Dirichlet and Neumann Laplacians on domains in $\R^n$ has been studied extensively in the literatures (see, e.g., \cite{Fuglede99, Davies00, BL08} and references therein). Less is known of spectral stability for the $\dbar$-Neumann Laplacian. In \cite{FuZhu19}, 
we studied spectral stability of the $\bar\partial$-Neumann Laplacian $\Box_q$ on a bounded domain $\Omega$ in $\C^n$ as the underlying domain is perturbed. We established upper semi-continuity properties for the variational eigenvalues of the $\bar\partial$-Neumann Laplacian on bounded pseudoconvex domains, lower semi-continuity properties on pseudoconvex domains that satisfy Catlin's property ($P$), and quantitative estimates on smooth bounded pseudoconvex domains of finite type in the sense of D'Angelo. In this paper, we consider the perturbation $\Box^t_q$ of the $\dbar$-Neumann Laplacian introduced by Kohn and Nirenberg \cite{KohnNirenberg65} in their elliptic regularization procedure. We study stability of the spectrum of $\Box^t_q$, first as $t\to 0^+$ and then as the underlying domain $\Omega$ is perturbed.

Unlike the classical Dirichlet or Neumann Laplacian, the spectrum of the $\dbar$-Neumann Laplacian need not be purely discrete (see \cite{FuStraube01} for an exposition on the subject). There are several ways to measure spectral stability under this circumstance. Here our focus is on stability of the variational eigenvalues defined by the min-max principle and convergence of the operators in resolvent sense (see Section~\ref{sec:prelim} below for the precise definitions). When the spectrum is purely discrete, the variational eigenvalues are indeed eigenvalues, arranged in increasing order and repeated according to multiplicity. Let $\lambda^q_k(\Omega)$ be the $k^{\rm th}$-variational eigenvalues of the $\dbar$-Neumann Laplacian $\Box_q$ on $(0, q)$-forms, $1\le q\le n-1$, on $\Omega$.  Let $\lambda^{t,q}_k(\Omega)$ be the $k^{\rm th}$-eigenvalue of $\Box_q^t$. Our first result concerns spectral stability of $\Box^t_q$ as $t\to 0^+$ (see Theorem~\ref{thm4} and Theorem~\ref{th:res} in Section~\ref{sec:er}):

 \begin{thm}\label{thm7} Let $\omz$ be a bounded domain in $\C^n$ with $C^2$ boundary. Let $1\le q\le n-1$ and $k\in\N$. Then	$\Box^t_q$ converges to $\Box_q$ in strong resolvent sense as $t\to 0^+$ and
 	\begin{equation}\label{tt}
 	\mathop{\lim}\limits_{t\to0^+}\lambda^{t,q}_k(\Omega)=\lambda^q_k(\Omega).
 	\end{equation}	
 	Furthermore, if $\omz$ is strongly pseudoconvex with smooth boundary, then $\Box^t_q$ converges to $\Box_q$ in norm resolvent sense and if $\omz$ is pseudoconvex of finite type in the sense of D'Angelo, then there exist positive constants $\alpha\in (0, \ 1/2]$ and $C$ independent of $t$ and $k$ such that
 	\begin{align}\label{t4}
 	|\lambda_k^{t,q}(\omz)-\lambda^q_k(\omz)|\le Ctk(\lambda^q_k(\omz))^{2([\frac{1}{2\alpha}]+1)},
 	\end{align}
 	where $[1/2\alpha]$ is the integer part of $1/2\alpha$.
  \end{thm}

 Our next result is about spectral stability of the Kohn-Nirenberg elliptic regularization operator $\Box^t_q$ as the underlying domain $\Omega$ is perturbed. Perturbation of the domain is measured in the $C^2$-topology.   Our main result in this direction is the following quantitative estimate:

 \begin{thm}\label{thmt}
	Let $\omz$ and $\Omega_j$ be smooth bounded pseudoconvex domains in $\C^n$ with normalized defining functions $r$ and $r_j$ respectively. Assume that $C^\infty$-norms of $r_j$ are uniformly bounded on $\ov\Omega_j$. Let $\delta_j=\|r-r_j\|_{C^2}$ be the $C^2$-norm  over $\ov\Omega\cup\ov\Omega_j$. 
Let $1\le q\le n-1$, $0<t<1$ and $k\in\N$.  Then there exist positive constants $\delta$ and $C_k$ such that
	\begin{align}\label{t2}
	\left|\lambda^{t,q}_k(\omz_j)-\lambda^{t,q}_k(\omz) \right|\le \dfrac{C_k\delta_j}{t^{2^{n+3}-1}},
	\end{align}
provided $\delta_j<\delta$.
\end{thm}

This paper is organized as follows. In Section~\ref{sec:prelim}, we recall the spectral theoretic
setup of the $\dbar$-Neumann Laplacian $\Box_q$ and the Kohn-Nirenberg elliptic regularization $\Box^t_q$. In Section~\ref{sec:er}, we study spectral stability of $\Box^t_q$ as $t\to 0^+$ and prove Theorem~\ref{thm7}. In Section~\ref{sec:deform}, we study spectral stability of $\Box^t_q$ as the underlying domain is perturbed and prove Theorem~\ref{thmt}. Throughout this paper, we will use $C$ to denote constants which might not be the same in different appearances.

\section{Preliminary}\label{sec:prelim}
In this section, we review the setup for the $\bar\partial$-Neumann
Laplacian (cf.~\cite{FollandKohn72, ChenShaw99}) and the elliptic regularization of the $\bar\partial$-Neumann
Laplacian(\cite{KohnNirenberg65}, see also \cite{Taylor96, Straube10}).  We define them
through their associated quadratic form. 

Let $\omz$ be a bounded domain in $\C^n$ and let $L^2_{(0, q)}(\Omega)$ be the space of $(0, q)$-forms with $L^2$-coefficients on $\Omega$ with respect to the 
standard Euclidean metric.  Let $\dbar_q\colon L^2_{(0, q)}(\Omega)\to L^2_{(0, q+1)}(\Omega)$ be the maximally defined Cauchy-Riemann operator.  The domain $\Dom(\dbar_q)$ of $\dbar_q$ consists of forms $u\in L^2_{(0, q)}(\Omega)$ such that $\dbar_q u\in L^2_{(0, q)}(\Omega)$ in the sense of distribution. Let $\dbarstar_q\colon L^2_{(0, q+1)}(\Omega)\to L^2_{(0, q)}(\Omega)$ be the adjoint of $\dbar_q$. Its domain is then given by
\begin{equation}\label{eq:dbarstar-dom}
\Dom(\dbarstar_q)=\big\{u\in L^2_{(0, q+1)}(\Omega) \mid \exists C>0,
|\langle u, \dbar_q v\rangle|\le C\|v\|,\ \forall v\in\Dom(\dbar_q)\big\}.
\end{equation}
When $\Omega$ has $C^1$-smooth boundary and  
$
u=\mathop{\sumprime}\limits_{|J|=q} u_J \,d\bar z_J \in C^1_{(0, q)}(\overline{\Omega}),
$
then $u\in\Dom(\dbarstar_{q-1})$ if and only if 
\begin{align}
   (\dbar r)^*\lrcorner u=	\sumprime_{|K|=q-1}\left(\sum_{k=1}^n u_{kK}\frac{\partial r}{\partial z_k}\right) d\bar z_K=0
\end{align}
on $\partial\Omega$, where $r$ is a defining function of $\partial\Omega$ such that $|\nabla r|=1$ on $\partial\Omega$ and 
\[
(\dbar r)^*=\sum_{j=1}^n \frac{\partial r}{\partial z_j}\frac{\partial}{\partial \bar z_j}
\]
is the dual $(0, 1)$-vector field of $\dbar r$ and $\lrcorner$ denotes the contraction operator. 
We decompose $u=u^{\tau}+u^{\nu}$ into the tangential part and normal part where
\[
u^{\nu}=((\dbar r)^*\lrcorner u)\wedge\dbar r \quad\text{and}\quad u^{\tau}=u-u^{\nu}.
\]

For $1\le q\le n-1$, let
\[
Q_q(u, v)=\langle\dbar_q u, \dbar_q
v\rangle_\Omega+\langle\dbarstar_{q-1} u,
\dbarstar_{q-1} v\rangle_\Omega
\]
be the sesquilinear form on $L^2_{(0, q)}(\Omega)$ with domain
$\Dom(Q_{q})=\Dom(\dbar_q)\cap \Dom(\dbarstar_{q-1})$. The $\dbar$-Neumann Laplacian $\square_q$ is 
the unique nonnegative self-adjoint operator $\square_{q}$ such that $Q_q(u, v)=\langle\square_q^{1/2} u, \square_q^{1/2}v\rangle_\Omega$ with $\Dom(\square_q^{1/2})=\Dom(Q_q)$.   
Consequently, $\square_q$ is given by
$$
\square_q=\dbar_{q-1}\dbarstar_{q-1}+\dbarstar_q\dbar_q
$$
and
$$	
\Dom(\square_q)=\{u\in L^2_{(0, q)}(\Omega) \mid u\in\Dom(Q_q), \dbar_q u\in \Dom(\dbarstar_{q}), \dbar_{q-1}^* u\in\Dom(\dbar_{q-1})\}.
$$
When $\omz$ is pseudoconvex, then it follows from  H\"{o}rmander's $L^2$-estimates for the $\dbar$-equation  that $\square_q$ has a bounded inverse $N_q=\Box_{q}^{-1}\colon L^2_{(0, q)}(\Omega)\to L^2_{(0, q)}(\Omega)$, the $\dbar$-Green's operator (\cite{Hormander65}, see also \cite{ChenShaw99}).

We now review the elliptic regularization in the setting of the $\dbar$-Neumann problem (\cite{KohnNirenberg65}). Let $\omz$ be a bounded domain in $\C^n$. For $t>0$, let
\[
Q_q^t(u,v)=Q_q(u,v)+t\langle\nabla u, \nabla v\rangle_\omz
\]
with $\Dom(Q_q^t)=W^1_{(0,q)}(\omz)\cap\Dom(\dbarstar_{q-1})$,  where the gradient operator $\nabla$ acts component-wise. (Hereafter, we use $W^s_{(0,q)}(\omz)$ to denote the space of $(0, q)$-forms with coefficients in the $L^2$-Sobolev space of order $s$.  The associated norm is denoted by either $\|\cdot\|_{W^s}$ or $\|\cdot\|_{s}$.) Then $Q_q^t$ is a densely defined, closed sesquilinear form on $L^2_{(0, q)}(\Omega)$. Let $\Box_q^t$ be the self-adjoint operator associated with $Q_q^t$. This is an elliptic operator with coercive boundary condition. It was introduced by Kohn and Nirenberg to study non-coercive boundary problems such as the $\dbar$-Neumann problem.  For abbreviation, we will call the operator $\Box_q^t$ the Kohn-Nirenberg Laplacian.

When $\partial\Omega$ is $C^2$-smooth,  then a form $u\in C^2_{(0,q)}(\ol\omz)$ belongs to $\Dom(\square_q^t)$ if and only if $u\in\Dom(\dbarstar_{q-1})$ and 
$$
(\dbar r)^*\lrcorner\dbar_q u+t\left(\dfrac{\partial u}{\partial \nu}\right)^{\tau}=0
$$ 
on $\partial\omz$, where $r$ is a $C^2$-smooth defining function of $\omz$ and $\left(\dfrac{\partial u}{\partial \nu}\right)^{\tau}$ is the tangent part of the (component-wise) normal derivative $\dfrac{\partial u}{\partial \nu}$
of $u$ (see \cite[\S~3.3]{Straube10} and \cite[Ch.~12]{Taylor96}). 

We will use $\lambda^q_k(\Omega)$ and $\lambda_k^{t,q}(\Omega)$ to denote the $k^{th}$-variational eigenvalues of $\square_q$ and $\square_q^t$ on $\Omega$ respectively, which are defined by min-max principle as follows:
 
 \begin{equation}\label{eq:minmax}
 \lambda^q_k(\Omega)=\inf_{L\subset\Dom(Q_q)\atop \dim L=k}\sup\limits_{u\in L\setminus\{0\}}\,Q_q(u,u)/\|u\|^2
 \end{equation}
 and 
  \begin{equation}
 \lambda^{t,q}_k(\Omega)=\inf_{L\subset\Dom(Q^t_q)\atop \dim L=k}\sup\limits_{u\in L\setminus\{0\}}\,Q^t_q(u,u)/\|u\|^2,
 \end{equation}
where the infima take over all linear $k$-dimension subspaces of $\Dom(Q_q)$ and  $\Dom(Q^t_q)$ respectively. Recall that the spectrum of a non-negative self-adjoint operator $S$ is purely discrete if and only if the variational eigenvalues $\lambda_k(S)$ defined as above goes to $\infty$ as $k\to\infty$. In this case, $\lambda_k(S)$ is the $k^{\text{th}}$-eigenvalue of $S$ when the eigenvalues are arranged in increasing order and repeated according to multiplicity (see \cite[Chapter~4]{Davies95}).  We collect some elementary properties of the Kohn-Nirenberg Laplacian in the following proposition:

\begin{prop}\label{prop:kn}  Let $\Omega$ be a bounded domain in $\C^n$. Let $k$ be a positive integer and let $K=k\cdot n!/q!(n-q)!$. Then
\begin{equation}\label{eq:kn1a}
	\lambda_k^q(\Omega)\le \lambda^{t, q}_k(\Omega)
\end{equation}
and
	\begin{equation}\label{eq:kn1b}
	t\lambda_k^N(\Omega)\le\lambda^{t, q}_K(\Omega)\le \big(\frac{1}{4}+t \big)\lambda^D_k(\Omega),
	\end{equation}
where $\lambda^N_k(\Omega)$ and $\lambda^D_k(\Omega)$ are respectively the $k^{\text{th}}$ variational
eigenvalues of the Neumann and Dirichlet Laplacians.  Furthermore, if $\partial\Omega$ is $C^1$-smooth, then $\Box^t_q$ has purely discrete spectrum and its first eigenvalue satisfies
\begin{equation}\label{kn1c}
\lambda^{t, q}_1(\Omega)\ge C\min\{t, 1\}.
\end{equation}
As a consequence, $N^t_q=(\Box^t_q)^{-1}$ is compact and satisfies
\begin{equation}\label{eq:kn2}
\|N^t_q u\|\le \big(1/C\min\{t, 1\}\big) \|u\|, \quad u\in L^2_{(0, q)}(\Omega)
\end{equation}	
for some constant $C$ independent of $t$. 
%Furthermore, if $\partial\Omega$ is smooth, then
%\begin{equation}\label{eq:kn3}
%\|N_q^tu\|_{s+2}\le C_s t\|u\|_{s},\quad u\in W^s_{(0,q)}(\omz)
%\end{equation}
%for some constant $C_s$ dependent on $s$ but independent of $t$, where $\|\cdot\|_{s}$ denotes the  $L^2$-Sobolev norm of order $s$. 
\end{prop}
\begin{proof} The inequality \eqref{eq:kn1a} is a consequence of the min-max principle in the definition of the variational eigenvalues and 
	the fact that
\[
\Dom(Q^t_q)\subset \Dom(Q_q)\quad\text{and}\quad  Q_q(u, u) \le Q^t_q(u, u), \quad u\in\Dom(Q^t_q).
\]
Since
\[
\Dom (Q^t_q)\subset W^1_{(0, q)}(\Omega)\quad\text{and}\quad 
t\|\nabla u\|^2\le Q^t_q (u, u), \quad u\in \Dom(Q^t_q),
\]
we have 
  \begin{equation}
  	\inf_{L\subset W^1_{(0, q)}(\Omega)\atop \dim L=K}\sup\limits_{u\in L\setminus\{0\}}\,\|\nabla u\|^2/\|u\|^2\le \lambda^{t, q}_{K}(\Omega).
  \end{equation}
The quantity on the left-hand side is the $K^{\text{th}}$-variational eigenvalues of the Neumann Laplacian acting componentwise on $(0, q)$-forms. We then obtain the first inequality in \eqref{eq:kn1b}.  Note that here we have used the fact that a $(0, q)$-form in $\C^n$ has  $n!/q!(n-q)!$ many components.
  
The second inequality in \eqref{eq:kn1b} follows similarly from the fact that
\[
W^1_{0, (0, q)}(\Omega)\subset \Dom (Q^t_q)
\]
and
\[
Q^t_q(u, u)=\big(\frac{1}{4}+ t\big)\|\nabla u\|^2, \quad u\in W^1_{0, (0, q)} (\Omega),
\] 
where $W^1_{0, (0, q)}(\Omega)$ is the completion of the space of smooth, compactly supported $(0, q)$-forms on $\Omega$ in $W^1_{(0, q)}(\Omega)$.

When $\Omega$ has $C^1$-smooth boundary, $W^1_{(0, q)}(\Omega)$ is relatively compact in $L^2_{(0, q)}(\Omega)$. It follows that
$
\{ u\in\Dom(Q^t_q) \mid \|u\|^2+Q^t_q(u, u)\le 1\}
$
is a relatively compact subset of $L^2_{(0, q)}(\Omega)$. Thus $\Box^t_q$ has compact resolvent and its spectrum is purely discrete. The smallest eigenvalue $\lambda_1^{t, q}(\Omega)$ of $\Box^t_q$ must be positive. Otherwise, if $\lambda_1^{t, q}(\Omega)=0$, then the corresponding eigenform $u$ satisfies $\|\nabla u\|=0$ and the $\dbar$-Neumann boundary condition $u\in\Dom(\dbarstar)$. Therefore $u$ has constant coefficients. Since $\partial\Omega$ is $C^1$-smooth, there are points on the boundary  where only one of the partial derivatives $\partial\rho/\partial z_j$, $1\le j\le n$, of a defining function $\rho$ of $\Omega$ is non-zero. (One can consider, for example, the points furthest from a coordinate hyperplane.)
By applying the $\dbar$-Neumannn boundary condition to $u$ on these points, we then conclude that the coefficients of $u$  must be all identically $0$, which leads to a contradiction. 

Since
\[
Q^t_q(u, u)\ge \min\{t, 1\} \big(Q_q(u, u)+\|\nabla u\|^2\big),
\]
we have
\[
\lambda^{t, q}_1 (\Omega)\ge C\min\{t, 1\}
\] 
where $C>0$ is the smallest eigenvalues of $\Box_q^{t_0}$ with $t_0=1$. Inequality~\eqref{eq:kn2} is then a consequence of the above inequality. %This is comes from the elliptic property of $\square_q^t$ (see \cite[Chapter~12]{Taylor96}).
\end{proof}

\begin{remark}\label{remark1} When $\Omega$ is pseudoconvex, it follows from H\"{o}rmander's $L^2$-estimates for the $\dbar$-operator that 
\begin{equation}\label{eq:h1}
Q^t_q(u, u)\ge Q_q(u, u)\ge \frac{q}{D^2 e} \|u\|^2, \quad u\in\Dom(Q^t_q) 
\end{equation}
and
\begin{equation}\label{eq:h2}
\|N^t_q u\|\le \frac{D^2 e}{q} \|u\|,
\end{equation}
where $D$ is the diameter of $\Omega$ (\cite{Hormander65}; see also \cite[Theorem~4.4.1]{ChenShaw99}).
\end{remark}

%Spectral stability of the complex Laplacian can be studied from different perspectives. Here our focus is on stability of the variational eigenvalues and the convergence in resolvent sense. 
%We will study spectral stability in the Kohn-Nirenberg Laplacian as $t\to 0^+$ and as the underlying domain $\Omega$ varies.  Perturbation of the domains is measured in part by the Hausdorff distance.  Recall that for two
%sets $A$ and $B$ in a metric space $(X, d)$, the Hausdorff distance between $A$ and $B$ is given by
%\[
%d_H(A,B)=\max\{\sup\limits_{x\in A}d(x,B),\sup\limits_{y\in B}d(y,A)\}.
%\]

Let $S_i$, $i=1, 2$, be non-negative self-adjoint operators on Hilbert spaces with associated quadratic forms $Q_i$. One way to estimate the difference of variational eigenvalues of $S_1$ and $S_2$ is to construct a transition operator $T\colon \Dom(Q_1)\to\Dom(Q_2)$ and estimate the differences between $\langle f, \ g\rangle_1$ and $\langle Tf, \ Tg\rangle_2$ and between $Q_1(f, g)$ and $Q_2(Tf, Tg)$ for any $f, g\in\Dom(Q_1)$. The following simple well-known lemma is useful (see, e.g., \cite[Lemma~2.1]{FuZhu19}). 

\begin{lem}\label{first} Let $k$ be a positive integer. Suppose 
 there exist $0<\alpha_k<1/(2k)$ and $\beta_k>0$ such that for any orthonormal set $\{u_1, u_2, \ldots, u_k\}\subset\Dom(Q_1)$,  
 \begin{equation}\label{eq:approx}
 |\langle Tu_h, Tu_l\rangle_{2}-\delta_{hl}|\le \alpha_k \quad\text{and}\quad |Q_{2}(T u_h, T u_l)-Q_{1}(u_h,u_l)|\le \beta_k.
 \end{equation}
 	Then 
  	\begin{align}\label{30}
 		\lambda_k(S_2)\le \lambda_k(S_1)+2 k(\alpha_k\lambda_k(S_1)+\beta_k).
 	\end{align}
 \end{lem}

\begin{remark}\label{remark2} Condition \eqref{eq:approx} in Lemma~\ref{first} can be replaced by the following: For any $k$-dimensional subspace $L_k$ of $\Dom(Q_1)$ and  $u\in L_k$,
\begin{equation}\label{eq:approx2}
	\|Tu\|^2_2\ge (1-k\alpha_k) \|u\|_1^2 \quad {\text{and}} \quad
	Q_2(Tu,Tu)\le Q_1(u, u)+k\beta_k \|u\|_1^2.
\end{equation}
We refer the reader to \cite{FuZhu19} for a proof of Lemma~\ref{first}.
\end{remark}

Spectral stability can also be studied from the perspective of resolvent convergence. Let $T_j$ and $T$ be self-adjoint operators on Hilbert space ${\mathbb H}$.  Recall that
$T_j$ is said to {\it converge to $T$ in norm resolvent sense}  if for all $\lambda\in {\mathbb C}\setminus {\mathbb R}$, the resolvent operator $R_\lambda(T_j)=(T_j-\lambda I)^{-1}$ converges to $R_\lambda(T)= (T-\lambda I)^{-1}$ in norm and $T_j$ is said to {\it converge to $T$ in strong resolvent sense}  if $R_\lambda(T_j)$ converges strongly to $R_\lambda(T)$.   It is well known that  if $T_j$ converges to $T$ in norm resolvent sense, then for any $\lambda\not\in\sigma(T)$,  $\lambda\not\in\sigma(T_j)$ for sufficiently large $j$, and if $T_j$ converges to $T$ in strong resolvent sense, then for any $\lambda\in\sigma(T)$, there exist $\lambda_j\in\sigma(T_j)$ so that $\lambda_j\to\lambda$. We refer the reader to \cite[\S VIII.7]{ReedSimon80} for relevant material.

\section{Spectral stability under the elliptic regularization}\label{sec:er}

In this section, we study spectral stability of the Kohn-Nirenberg Laplacian $\Box^t_q$ as $t\to 0^+$. We obtain quantitative estimates for the difference between $\lambda^q_k(\Omega)$ and $\lambda^{t,q}_k(\Omega)$ when $\omz$ is a smooth bounded pseudoconvex domain of finite type in the sense of D'Angelo. We also study the convergence of $\square_q^t$ in resolvent sense as $t\to 0^+$.  

A notion of finite type was introduced by D'Angelo \cite{Dangelo82} in connection with subelliptic theory of the $\dbar$-Neumann Laplacian. Roughly speaking, the $D_q$-type of a smooth bounded domain $\Omega$ is the maximal order of contact of $\partial\Omega$ with any $q$-dimensional complex analytic variety. (We refer the readers to  \cite{Dangelo82, Dangelo93, DangeloKohn99} for the precise definition.) Catlin \cite{Catlin83, Catlin87} showed that a smooth bounded pseudoconvex domain $\Omega$  is of finite $D_q$-type if and only if there exist constants $0<\alpha\le 1/2$ and $C>0$ such that the following subelliptic estimate holds:
\begin{equation}\label{subelliptic}
\|u\|^2_{\alpha}\le C Q_q(u, u), \quad u\in\Dom(Q_q).
\end{equation}
The constant $\alpha$ is usually referred to as the order of subellipticity for the $\dbar$-Neumann Laplacian and it is equal to $1/2$ when $\omz$ is strongly pseudoconvex. The following lemma is a direct consequence of Catlin's theorem. 

\begin{lem}\label{lem15} 
Let $\omz$ be a smooth bounded pseudoconvex domain of finite $D_q$-type in $\C^n$. Let $u$ be an eigenform of the $\dbar$-Neumann Laplacian $\Box_{q}$ with associated eigenvalue $\lambda(\Omega)$. Let $m$ and $l$ be non-negative integers.  Then there exist positive constants $\alpha\in (0, \ 1/2]$,  $B_m$ and $C_l$ such that 
	\begin{align}
	\|u\|_{W^{2m\alpha}}&\le B_m(\lambda(\omz))^{m}\|u\| \label{28a}
	\intertext{and}
	\|u\|_{C^l(\overline{\Omega})}&\le C_l (\lambda(\omz))^{[\frac{n+l}{2\alpha}]+1} \|u\|. \label{28}
	\end{align}
	$($Hereafter we use $[a]$ to denote the integer part of a real number $a$.$)$
\end{lem}
\begin{proof}  From above-mentioned work of Catlin, we know that there exist constants $\alpha\in (0, \ 1/2]$ and $C_s>0$ such that
\begin{equation}\label{cat-est}
\|N_q u\|_{s+2\alpha}\le C_s\|u\|_{s}.
\end{equation}
Starting with $s=0$ and repeatedly applying \eqref{cat-est} to $\Box_\Omega u=\lambda(\Omega) u$, we obtain \eqref{28a}. The estimate \eqref{28} is then an immediate consequence of the Sobolev embedding theorem.
\end{proof}

Recall that $C^1_{(0, q)}(\ol{\Omega})\cap\Dom(\dbarstar_{q-1})$ is dense in $\Dom(Q_q)$ in the graph norm 
\[
\|u\|_Q=(\|u\|^2_\Omega +Q_q(u, u))^{1/2}
\]
when $\partial\omz$ is $C^2$-smooth (see, e.g., \cite[Lemma~4.3.2]{ChenShaw99}). Thus $\Dom(Q^t_q)$ is also dense in $\Dom(Q_q)$ in the graph norm. We will use this fact in proving the following Theorem:

 \begin{thm}\label{thm4} Let $\omz$ be a bounded domain in $\C^n$ with $C^2$-smooth boundary. Let $1\le q\le n-1$ and $k\in\N$. Then	
    \begin{equation}\label{tto}
     \mathop{\lim}\limits_{t\to0^+}\lambda^{t,q}_k(\Omega)=\lambda^q_k(\Omega).
    \end{equation}	
 Furthermore, if $\omz$ is a smooth bounded pseudoconvex domain of finite $D_q$-type, then there exist positive constants $\alpha\in (0, \ 1/2]$ and $C$ independent of $t$ or $k$ such that
     \begin{align}\label{t1}
 	 |\lambda_k^{t,q}(\omz)-\lambda^q_k(\omz)|\le C tk (\lambda^q_k(\omz))^{2([\frac{1}{2\alpha}]+1)}.
 	 \end{align}
  \end{thm}

\begin{proof} On the one hand, from Proposition~\ref{prop:kn} we know that $\lambda_k^q (\Omega)\le\lambda_k^{t, q}(\Omega)$. On the other hand,  since $W^1_{(0,q)}(\omz)\cap\Dom(\dbarstar_{q-1})$ is dense in $\Dom(Q_q)$ in the graph norm $\|\cdot\|_Q$,    for any $\varepsilon>0$, there exists $k$-dimensional subspace $L_k\subset W^1_{(0,q)}(\omz)\cap\Dom(\dbarstar_{q-1})$ such that
	\begin{equation}\label{xx}
	\begin{aligned}
	\lambda^q_k(\omz)+\varepsilon&\ge\lambda_Q(L_k)=\sup\limits_{u\in L_k\setminus\{0\}}\dfrac{Q^t(u,u)-t\|\nabla u\|^2}{\|u\|^2}\\
	&\ge \sup\limits_{u\in L_k\setminus\{0\}}\dfrac{Q^t(u,u)}{\|u\|^2}-t\sup\limits_{u\in L_k\setminus\{0\}}\dfrac{\|\nabla u\|^2}{\|u\|^2}\\
	&\ge\lambda_k^{t,q}(\omz) -t\sup\limits_{u\in L_k\setminus\{0\}}\dfrac{\|\nabla u\|^2}{\|u\|^2}.
	\end{aligned}
	\end{equation}
	Letting $t\to 0^+$, we then have $\limsup\limits_{t\to0}\lambda_k^{t,q}(\omz)\le \lambda_k^{q}(\omz)$ and hence (\ref{tto}).
	
	Under the pseudoconvexity and finite type assumptions, the spectrum of $\Box_q$ is purely discrete. 
	Let $u_l$ be eigenforms associated with eigenvalues $\lambda^q_l(\omz)$, $1\le l\le k$.
	Let $L_k=\Span\{u_1, u_2, \cdots, u_k\}$ and let $u\in L_k$. It follows from Lemma \ref{lem15} and the Cauchy-Schwarz inequality that there exist constants $\alpha\in (0, \ 1/2]$ and $C>0$ such that
 	\begin{align*}
 	\|\nabla u\|^2\le Ck(\lambda^q_k(\omz))^{2([\frac{1}{2\alpha}]+1)}\|u\|^2.
 	\end{align*}
 	Thus
 	$$0\le\lambda_k^{t,q}(\omz)-\lambda^q_k(\omz)\le t\sup\limits_{u\in L_k\setminus\{0\}}\,\|\nabla u\|^2/\|u\|^2\le C tk (\lambda^q_k(\omz))^{2([\frac{1}{2\alpha}]+1)}.$$
This concludes the proof of Theorem~\ref{thm4}. \end{proof}	
 
% Let $\Omega$ be bounded pseudoconvex domains in $\C^n$. Then the spectra of $\Box_{q}$ and $\Box^t_{q}$ are uniformly bounded away from $0$. It is easy to see that $\Box^t_{q}$ converge to $\Box_{q}$ in strong resolvent sense if $N^t_q$ converges to $N_q$ strongly on $L^2_{(0, q)}(\omz)$ (see, e.g., \cite[Theorem VIII.9]{ReedSimon80}). 

We now study the resolvent convergence of the Kohn-Nirenberg Laplacian $\Box^t_q$ as $t\to 0^+$. Our result is as follows.

 \begin{thm}\label{th:res}  Let $\omz$ be a bounded domain in $\C^n$ with $C^2$-smooth boundary. Then $\square^{t}_q$ converges to $\square_q$ in  strong resolvent sense. If $\Omega$ is strongly pseudoconvex with smooth boundary, then $\square^t_q$ converges to $\square_q$ in norm resolvent sense. 
\end{thm}

\begin{proof} Let $\widetilde{Q}_q(u, v)=Q_q(u, v)+\langle u, v\rangle$ and let 
	$$
	F_q=\Box_q+I \quad \text{and}\quad R_q=(\Box_q+I)^{-1}. 
	$$
	Similarly, let  $\widetilde{Q}^t_q(u, v)=Q^t_q(u, v)+\langle u, v\rangle$ and let 
	$$
	F^t_q=\Box^t_q+I \quad \text{and}\quad R^t_q=(\Box^t_q+I)^{-1}.
	$$ 
For any $u\in L^2_{(0, q)}(\Omega)$, $R_q(u)\in\Dom(\Box_q)\subset\Dom(Q_q)$. Since $C^1_{(0, q)}(\overline{\Omega})\cap\Dom(\dbarstar_{q-1})$ is dense in $\Dom(Q_q)$ in the graph norm $\|\cdot\|_Q$,
for any $\epsilon>0$, there there exists $v\in W^1_{(0,q)}(\omz)\cap\Dom(\dbarstar_{q-1})$ such that 
$$
\|v- R_qu\|_{Q}<\eps.
$$
Note that $R^t_q(u)\in \Dom(\Box^t_q)\subset \Dom(Q^t_q)\subset\Dom(Q_q)$. We have
	 \begin{align*}
	 \|R_q^t u-R_qu\|^2&\le \widetilde{Q}_q(R_q^t u-R_qu, R_q^t u-R_qu)\\
	 &=\widetilde{Q}_q(R_q^t u, R_q^t u-R_qu)-\widetilde{Q}_q(R_q u, R_q^t u-R_qu)\\
	 &=\widetilde{Q}_q(R_q^t u, R_q^t u-v)+\widetilde{Q}_q(R_q^t u, v-R_qu)-\langle u, R_q^t u-R_qu\rangle\\
	 &=\langle u, R_qu-v\rangle-t\langle\nabla R_q^t u, \nabla(R_q^t u-v)\rangle+\widetilde{Q}_q(R_q^t u,v-R_qu)\\
	 &\le\langle u, R_qu-v\rangle+t\langle\nabla R_q^t u, \nabla v\rangle+\widetilde{Q}_q(R_q^t u,v-R_qu)\\
	 &\le \|u\|\|R_q u-v\|+t\|\nabla R_q^t u\|\|\nabla v\|+\|R^t_qu\|_{Q}\|v- R_qu\|_{Q}.
	 \end{align*}
	 Note that
	 \begin{align*}
	 t\|\nabla (R_q^t u)\|^2\le \widetilde{Q}_q^t(R_q^t u, R_q^t u)=\langle u, R_q^t u\rangle\le \|u\|\|R_q^t u\|\le \|u\|^2
	 \end{align*}
	 and 
	  \begin{align*}
	  	\|R_q^t u\|^2_Q\le \widetilde{Q}_q^t(R_q^t u, R_q^t u)=\langle u, R_q^t u\rangle\le \|u\|\|R_q^t u\|\le \|u\|^2.
	  \end{align*}
	  It follows that
	 \begin{align*}
	 \|R_q^t u-R_qu\|^2\le \|u\|\left(t^{1/2}\|\nabla v\|+2\eps\right).
	 \end{align*}
	Letting $t\to 0^+$ and then $\eps\to 0^+$, we then conclude that $\|R^t_q u-R_q u\|\to 0$ as $t\to 0^+$.
	
	When $\Omega$ is pseudoconvex, the spectra of $\Box_q^t$ and $\Box_q$ are both contained in the interval $[q/eD^2,\ \infty)$, where as before $D$ is the diameter of $\Omega$ (\cite{Hormander65}; see also, e.g., \cite[Theorem~4.4.1]{ChenShaw99}). Thus in this case, it suffices to consider the convergence of the $\dbar$-Green operator $N^t_q$ (see, e.g., \cite[Theorem~VIII.9]{ReedSimon80}).  
	When $\Omega$ is strongly pseudoconvex with smooth boundary, from Kohn's subelliptic estimate we 
	know that there exists a constant $C>0$ such that 
	\[
	 \|N_q u\|_1 \le C\|u\| \quad \text{and}\quad \|N^t_q u\|_1\le C\|u\|
	 \]
	 for any $u\in L^2_{(0, q)}(\Omega)$ (\cite{Kohn63, FollandKohn72}). Thus $N_q u\in\Dom(Q^t_q)$ and we have
	 \begin{align*}
	 Q_q(N_q^t u,N_q^t u-N_qu) &=Q^t_q(N_q^t u,N_q^t u-N_q u)-t\langle\nabla N_q^t u,\nabla(N_q^t u-N_q u)\rangle \\
	 &=\langle  u, N^t_q u-N_qu\rangle- t\langle\nabla N_q^t u,\nabla(N_q^t u-N_q u)\rangle \\
	 &=Q_q(N_q u,N_q^t u-N_qu)- t\langle\nabla N_q^t u,\nabla(N_q^t u-N_q u)\rangle. 
	 \end{align*}
	 Therefore
	\begin{align*}
	\dfrac{q}{eD^2}\|N_q^t u-N_qu\|^2&\le Q_q(N_q^t u-N_qu,N_q^t u-N_qu)\\
	&=Q_q(N_q^t u,N_q^t u-N_qu)-Q_q(N_q u,N_q^t u-N_qu)\\
	&=-t\langle\nabla N_q^t u,\nabla(N_q^t u-N_q u)\rangle\le t\langle\nabla N_q^t u,\nabla N_q u\rangle \\ 
	&\le t\|\nabla N_q^t u\|\|\nabla N_q u\| \le Ct\|u\|^2.
	\end{align*}
	Hence $N^t_q$ converges to $N_q$ in norm as $t\to 0^+$.
\end{proof}
 
\begin{remark}\label{remark3} One cannot expect that $\Box^t_q$ converges to $\Box_q$ in norm resolvent sense if $\Omega$ is only assumed to be weakly pseudoconvex with smooth boundary. For example, if $\partial\Omega$ contains an $(n-1)$-dimensional complex analytic variety, then by Proposition~\ref{prop:kn}, $N^t_q$ is compact but $N_q$ is not (see \cite{FuStraube01}).  Hence $N^t_q$ cannot converge to $N_q$ in norm.  
\end{remark}

\section{Spectral stability under domain perturbation}\label{sec:deform} 

Our aim in this section is to establish a quantitative estimate for   $|\lambda^{t,q}_k(\omz_1)-\lambda^{t, q}_k(\omz_2)|$ when $\omz_1$ and $\omz_2$ are smooth bounded domains in $\C^n$ that are sufficiently close to each other.  The key is to construct a transition operator $T$ form $\Dom(Q^t_{q,\omz_1})$ to $\Dom(Q^t_{q,\omz_2})$ such that $|\|Tu\|_{\omz_2}-\|u\|_{\omz_1}|$ and $|Q^t_{q,\omz_2}(Tu,Tu)-Q^t_{q,\omz_1}(u,u)|$ is controlled
by the closeness between $\Omega_1$ and $\Omega_2$.  (Here we use $Q^t_{q,\omz}$ to denote the quadratic form associated with $\Box^t_q$ acting on $(0,q)$-forms on $\omz$. To economize the notation, we will sometimes drop the subscript $q$ when doing so causes no confusion.) Since $\Dom(Q^t_{\omz})=W_{(0,q)}^1(\omz)\cap\Dom(\dbarstar)$, the restriction of a form from $\Dom(Q_\omz^t)$ no longer belongs to $\Dom(Q^t_{U})$ where $U$ is a subdomain of $\omz$. Additionally, the extension of a form from $\Dom(Q_\omz^t)$ to zero outside $\omz$ does not make it belong to $\Dom(Q^t_{V})$ where $V$ is a larger domain containing $\omz$. As in \cite{FuZhu19}, we  overcome these difficulties by decomposing $u\in \Dom(Q^t_\Omega)$ into the tangential and normal components and treat them separately. The tangential component is dealt with as in the case of the Neumann Laplacian while the normal component is handled as in the case of the Dirichlet Laplacian.

We now elaborate on how to measure the closeness between domains.  Let $\Omega$ be a bounded domain in $\C^n$ with $C^m$-smooth boundary ($2\le m\le \infty$).
A real valued function $r\in C^m(\C^n)$ is said to be a defining function of $\Omega$ if
$r<0$ on $\Omega$, $r>0$ on $\C^n\setminus\overline{\Omega}$, and $|\nabla r|\not=0$ on $\partial\Omega$. The defining function is {\it normalized} if $|\nabla r|=1$ on $\partial\Omega$. 
Let $\rho$ be the signed distance function of $\Omega$ such that $\rho(z)=-\dist(z, \partial\Omega)$ when $z\in\Omega$ and $\rho(z)=\dist(z, \partial\Omega)$ when $z\in\C^n\setminus\Omega$. It is well known that there is a neighborhood $U$ of $\partial\Omega$ such that $\rho\in C^m(U)$ (see \cite{KrantzParks81}). It follows that for any normalized defining function $r(z)$ of $\Omega$, we have $r(z)=h(z)\rho(z)$ for some positive function $h\in C^{m-1}(U)$ such that $h=1$ on $\partial\Omega$. For $\delta>0$, let
\[
\Omega^{-}_\delta=\{z\in\C^n \mid r(z)<-\delta\}\quad\text{and}\quad\Omega^{+}_\delta=\{z\in\C^n \mid r(z)<\delta\}.
\]
Let $\Omega_j$ be a bounded domain in $\C^n$ with $C^m$-smooth boundary. Let $r_j$ be a normalized defining function for $\Omega_j$.  The closeness between $\Omega$ and $\Omega_j$ will be measured  by $\delta_j=\|r-r_j\|_{C^2}$,  the $C^2$-norm of $r-r_j$ over $\overline\Omega\cup\overline\Omega_j$. Note that for any $a>1$, 
\[
\Omega^-_{a^{-1}\delta_j}\subset \Omega_j\subset \Omega^+_{a\delta_j} 
\]
provided $\delta_j$ is sufficiently small. Furthermore, the signed distance function $\rho_j$ of $\Omega_j$ is $C^m$ on some neighborhood $U$ of $\partial\Omega$ (see \cite[Lemma~4.11]{Federer59} and \cite[Theorem~3]{KrantzParks81}).  

We first establish some auxiliary estimates.  It follows from the elliptic theory that $\|u\|_{s+2}\le C_t\|\square^t u\|_{s}$ on a smooth bounded pseudoconvex domain (see \cite[Proposition~3.5]{Straube10}). The following lemma is a quantitative version of this result. Throughout this section, we will assume that $0<t<1$, $1\le q\le n-1$  and $n$ is a positive integer.    
   \begin{lem}\label{ttt}
  	Let $\omz$  be a smooth bounded domain in $\C^n$. Let $s$ be a non-negative integer. Then there exists a constant $C>0$ independent of $t$ such that
  	\begin{align}\label{32}
  	\|u\|_{W^{s+2}}\le \dfrac{C}{t^{2^{s+1}}}\|\square_q^t u\|_{W^{s}}, 
  	\end{align}
for all $u\in \Dom(\square_q^t)$ with $\square_q^t u\in W^s_{(0,q)}(\omz)$. Moreover, if $\Omega$ is pseudoconvex, then the above estimate can be improved as follows:
 	\begin{align}\label{32b}
 		\|u\|_{W^{s+2}}\le \dfrac{C}{t^{3\cdot 2^{s-1}}}\|\square_q^t u\|_{W^{s}}.
 	\end{align}
  \end{lem}
  
  \begin{proof} The proof follows the same line of arguments as in the proof of Proposition~3.5 in~\cite{Straube10}. One just needs to keep track of the constants. We provide the details for  the proof of \eqref{32}. By Proposition~\ref{prop:kn}, we have
  	\begin{equation}\label{32c}
  	Ct\|u\|^2\le Q^t(u, u), \quad u\in\Dom(Q^t),
  	\end{equation}
  		where the constant $C>0$ is independent of $t$.
  Hence  
  	\begin{equation}\label{32d} 		
  	\|u\|_{W^{1}(\omz)}\le \dfrac{C}{\sqrt t}(Q^{t}(u,u))^{1/2}\le\dfrac{C}{\sqrt t}\|\square^t u\|^{1/2}\|u\|^{1/2},\quad u\in\Dom(\square^t).
\end{equation}
It follows from \eqref{eq:kn2} that $\|u\|\le (C/t) \|\Box^t u\|$. Therefore
  	\begin{equation}\label{33}
  	\begin{aligned}
  	\|u\|_{W^{1}}\le\dfrac{C}{t}\|\square^t u\|.
  	\end{aligned}
  	\end{equation}

It suffices to prove (\ref{32}) when $u$ is supported in a special boundary chart. The general case is obtained by a partition of unity argument. 
Let $(t_1, \ldots, t_{2n-1}, r)$ be a local special coordinate chart near a boundary point where $r$ is a defining function of $\Omega$ and  $(t_1,\cdots, t_{2n-1})$ are coordinates on the boundary. Denote by $D^h_j$, $1\le j\le 2n-1$, the difference quotient with respect to $t_j$, acting on forms coefficientwise in a special boundary frame associated to special boundary chart. Note that $D^h_j$ preserves $\Dom(\dbarstar)$. For any $(0,q)$-form $v\in\Dom(Q^t)$, we have
    \begin{align*}
    &|\langle\square^t D_j^h u,v\rangle|= |\langle\dbar D_j^h u, \dbar v\rangle+\langle\dbarstar D_j^h u, \dbarstar v\rangle+t\langle\nabla D_j^h u, \nabla v\rangle|\\
    &\qquad\le C\|u\|_{W^{1}}\|v\|_{W^{1}}+|\langle D_j^h\dbar  u, \dbar v\rangle+\langle D_j^h\dbarstar u, \dbarstar v\rangle+t\langle D_j^h\nabla u, \nabla v\rangle|\\
    &\qquad= C\|u\|_{W^{1}}\|v\|_{W^{1}}+ |\langle \dbar  u, D_j^{-h}\dbar v\rangle+\langle \dbarstar u, D_j^{-h}\dbarstar v\rangle+t\langle \nabla u, D_j^{-h}\nabla v\rangle|\\
    &\qquad\le C\|u\|_{W^{1}}\|v\|_{W^{1}}+|Q^t(u, D_j^{-h}v)|\le\dfrac{C}{t}\|\Box^t u\|\|v\|_{W^{1}}.
    \end{align*}
   Substituting  $v$ by $D_j^h u$ in the above estimate, we obtain
    \begin{align*}
    	\|D_j^h u\|^2_{W^{1}}\le \dfrac{C}{t}Q^t(D_j^h u,D_j^h u)=\frac{C}{t}\langle \Box^t D^h_j u, \Box^t D^h_j u\rangle\le \dfrac{C}{t^2}\|\square^t  u\|\|D_j^h u\|_{W^{1}}.
    \end{align*}
    Therefore
    \begin{equation}\label{34}
    \begin{aligned}
    \|\nabla_T u\|_{W^{1}}\le C\|D_j^h u\|_{W^{1}}\le\dfrac{C}{t^2}\|\square^t  u\|,
    \end{aligned}
    \end{equation}
    where $\nabla_T$ denotes the gradient with respect to the tangential coordinates $(t_1,\cdots, t_{2n-1})$.
    
    We now estimate the full Sobolev norm.  Note that  
    $$\square^t u=-(\dfrac{1}{4}+t)\Delta u$$ 
    when $u\in\Dom(\square^t)$. Writing $\Delta$ in terms of tangential and normal derivatives in the local coordinates, we obtain
    $$\left\|\dfrac{\partial^2u}{\partial\nu^2}\right\|\le C\left(\|\square^t u\|+\|u\|_{W^{1}}+\|\nabla_T u\|_{W^{1}}\right),$$
    where $\dfrac{\partial}{\partial\nu}$ is the normal derivative.
    Consequently,
    \begin{equation}\label{35}
    \begin{aligned}
    \|u\|_{W^{2}}\le C\left(\|\nabla_T u\|_{W^{1}}+\|\partial^2u/\partial\nu^2\|\right)\le\dfrac{C}{t^2}\|\square^t  u\|.
    \end{aligned}
    \end{equation}
    Thus \eqref{32} holds for $s=0$. We proceed with the inductive step. Assume that \eqref{32} holds for $s$. Then
    
    \begin{equation}
    \begin{aligned}
    \|D_j^hu\|_{W^{s+2}}&\le \dfrac{C}{t^{2^{s+1}}}\|\square^t D_j^hu\|_{W^{s}}\\
    & \le \dfrac{C}{t^{2^{s+1}}}\left(\|u\|_{s+2}+\| D_j^h\square^tu\|_{W^{s}}\right)\\
    &\le \dfrac{C}{t^{2^{s+2}}}\|\square^tu\|_{W^{s+1}}.
    \end{aligned}
    \end{equation}
    Thus 
    $$\|\nabla_Tu\|_{W^{s+2}}\le \dfrac{C}{t^{2^{s+2}}}\|\square^tu\|_{W^{s+1}}.$$
  By the same arguments proceeding (\ref{35}), we then establish \eqref{32} for $s+1$. 
  
  When $\Omega$ is pseudoconvex, from H\"ormander's $L^2$-estimate, we have $\|u\|\le C\|\Box^t u\|$. Using this instead of \eqref{eq:kn2}, we obtain
  \begin{equation}\label{35b}
  	\|u\|_{W^1}\le \frac{C}{\sqrt{t}}\|\Box^t u\|.
 \end{equation} 	
 Note that in this case, the constants in \eqref{35b} and \eqref{32c} depends only on the diameter of $\Omega$ and $q$. The rest of the proof follows from the same lines except with different exponents of $t$. This concludes the proof of Lemma~\ref{ttt}. \end{proof}
 
 \begin{lem}\label{infty}
    Let $\omz$  be a smooth bounded domain in $\C^n$.  Let $u$ be an eigenform of $\Box^t$ with associated eigenvalue $\lambda^t(\Omega)$. Let $l$ be a non-negative integer. Then there exists a constant $C>0$ independent of $t$ such that
 	\begin{align}\label{9a}
 	\|u\|_{C^l(\overline{\Omega})}\le \dfrac{C}{t^{2(2^{n+l+2}-1)/3}}(\lambda^{t}(\omz))^{[\frac{n+l}{2}]+1}\|u\|,
 	\end{align}	
where $[(n+l)/2]$ as before denotes the integer part of $(n+l)/2$. Furthermore, if $\Omega$ is pseudoconvex, then
 	\begin{align}\label{9}
 		\|u\|_{C^l(\overline{\Omega})}\le \dfrac{C}{t^{(2^{n+l+2}-1)/2}}(\lambda^{t}(\omz))^{[\frac{n+l}{2}]+1}\|u\|.
 	\end{align}	
 \end{lem}
 \begin{proof} This is a direct consequence of Lemma~\ref{ttt} and the Sobolev embedding theorem. We provide only the proof for \eqref{9}. 
    
    For $u\in \Dom(\square^t)$, from (\ref{32}) with $s=0$, we have 
    $$\|u\|_{W^2}\le \dfrac{C}{t^{3/2}}\|\square^t u\|=\dfrac{C}{t^{3/2}}\lambda^{t}(\omz)\|u\|.$$
    Thus $u\in W^2_{(0,q)}(\omz)$ and $\square^t u=\lambda^t(\omz) u\in W^2_{(0,q)}(\omz)$. From \eqref{32} with $s=2$, we obtain
    $$\|u\|_{W^4}\le \dfrac{C}{t^6}\|\square^t u\|_{W^2}\le \dfrac{C}{t^{15/2}}(\lambda^{t}(\omz))^{2}\|u\|.
    $$
    Repeating this process, we obtain  $u\in W^{2m}_{(0,q)}(\omz)$ and 
    $$\|u\|_{W^{2m}(\omz)}\le \dfrac{C}{t^{(2^{2m}-1)/2}}(\lambda^{t}(\omz))^{m}\|u\|,\quad m\in\N.$$
    The desired inequality (\ref{9}) is then an immediate consequence of Sobolev embedding theorem.
 \end{proof}
 
Let $\Omega$ be a pseudoconvex domain in $\C^n$ with $C^m$-smooth boundary ($m\ge 2$). Let $r(z)$ be a normalized defining function of $\Omega$. Then $|\nabla r(z)|=1$ on $\partial\Omega$. Let $z'\in\partial\omz$ and let $U'\subset U$ be a tubular neighborhood of $z'$ such that $|\nabla r(z)-\nabla r(z')|<1/2$ when $z\in U'$ and 
$$
\Omega^\pm_\delta=\{z\in\C^n \mid r(z)<\pm\delta\}.
$$ 
Shrinking $U'$ if necessary, then for sufficiently small $\delta>0$, we have $z-2\delta \vv n(z')\in \Omega$ for all $z\in U'\cap \Omega^+_\delta$ and $z+2\delta\vv n(z')\not\in\Omega$ for all $z\in U'\setminus\Omega^-_\delta$. Furthermore, 
$$\dist(z- 2\delta\vv n(z'),\partial\omz)\ge\dist(z - 2\delta\vv n(z),\partial\omz)-2\delta|\vv n(z)-\vv n(z')|>2\delta-\delta=\delta,$$
for all $z\in U'\cap\Omega^+_\delta$. 
We choose a finite covering $\{V^l\}_{l=0}^{m}$ of $\ol{\omz}$ such that $V^0\subset\subset\omz$ and each $V^l$, $1\le l\le m,$ is a tubular neighborhood about some $z^l\in\partial\omz$ constructed as above. Write $\vv n^l=\vv n(z^l)$. We then have
$$
\bigcup_{l=1}^{m}\left\lbrace z-2\delta\vv n^l\,|\,z\in V^l\cap\omz\right\rbrace \bigcup V^0 \subset   \omz^-_\delta
$$
and 
$$
\bigcup_{l=1}^{m}\left\lbrace z+2\delta\vv n^l\,|\,z\in V^l\cap\omz\right\rbrace\bigcup V^0 \supset   \omz^+_\delta.
$$

We now construct a frame for $(0, 1)$-forms on $V^l$. Since $|\nabla r|>1/2$ on $V^l$. Shrinking $V^l$ if necessary, we may assume without loss of generality that $\partial r/\partial z_1\not=0$ on $V^l$. Let 
\[
\overline\omega^n_l=\bar\partial r \quad\text{and}\quad \overline\omega^k_l=d\bar z_k-\big(4\partial r/\partial z_k\big)\bar\partial r, \ \  1\le k\le n-1.
\]
Note that since $|\nabla r|=1$ on $\partial\Omega$.  Hence $\overline\omega^k_l$, $1\le k\le n-1$, is pointwise orthogonal to $\overline\omega^n_l$ on $\partial\Omega$ and satisfies the $\dbar$-Neumann boundary condition on $V^l\cap\partial\Omega$. Furthermore, the determinant of the coefficients of the $(0, 1)$-forms $\overline{\omega}^k_l$, $1\le k\le n$, is non-zero on $V^l$. Thus $\{\overline{\omega}^1_l, \ldots,  \overline{\omega}^n_l\}$ is indeed a frame for the $(0, 1)$-forms on $V^l$.  Let $\Omega_j$ be a bounded domain with $C^m$-boundary with a normalized defining function $r_j$ such that $\|r_j-r\|_{C^2}$ is sufficiently small. We then construct a frame $\{ \ov{\omega}^k_{j, l} ; \ 1\le k\le n\}$ as above but with $r$ replaced by $r_j$. Thus we have
\begin{equation}\label{eq:f1} 
\|\ov{\omega}^k_l -\ov{\omega}^k_{j, l}\|_{C^1(V^l\cap(\Omega_j\cup\Omega))}\le C\|r-r_j\|_{C^2(\Omega_j\cup\Omega)}.
\end{equation}

Let $\{\psi^l\}_{l=0}^{m}$ be a partition of unity subordinated the covering $\{V^l, \ 0\le l\le m\}$ such that $\supp \psi^l\subset V^l$. Let $\ec \colon W^s(\Omega)\to W^s(\C^n)$ be a continuous extension operator. Recall that the norm of this operator depends only
on $n$, $s$, and the Lipschitz constant of $\Omega$ (\cite[Ch~VI.3, Theorem~5]{Stein70}). Let $d(z)=\dist(z,\partial\omz)$. Let $\chi_j(t)$ be a smooth function such that $\chi_j(t)=0$ if $t>2\delta_j$, $\chi_j(t)=2\delta_j$ if $t<\delta_j$, and $0\le \chi'_j(t)\le 2$. 

We are now in position to define the transition operator.  Let $u^t\in C^\infty_{(0,q)}(\overline\omz)\cap \Dom(Q^t_\omz)$.  Using the partition of unity, we write 
\begin{equation}
	\begin{aligned}
		u^t=\psi^0 u^t+\sum_{l=1}^{m}\mathop{{\sum}'}\limits_{|J|=q}\psi^l u^t_J \bar\omega_{l}^J,
	\end{aligned}
\end{equation}
where $\{\bar\omega_{l}^1,\cdots,\bar\omega_{l}^n\}$ is the frame for $(0, 1)$-forms on $V^l$ ($1\le l\le m$) constructed as above. Since $u^t\in\Dom(Q^t_q)$, we have $u^t_J=0$ on $\partial\Omega$ when $n\in J$. We extend such $u^t_J$'s to be zero outsite of $\Omega$. Define $T_j:C^\infty_{(0, q)}(\ov{\Omega})\cap\Dom(Q^t_\omz)\to C^1_{(0, q)}(\ov{\Omega})\cap\Dom(Q^t_{\omz_j})$ by
\begin{equation}
	\begin{aligned}
		T_ju^t=\psi^0u^t+\sum_{l=1}^{m}\Big(\sumprime\limits_{J,n\notin J}\ec[\psi^lu^t_J]\bar\omega_{j,l}^J+\sumprime\limits_{J,n\in J} \psi^l(z) u^t_J(z+\chi(d(z))\vv n^l)\bar\omega_{j,l}^J\Big).
	\end{aligned}
\end{equation}
where $\{\bar\omega_{j,l}^1,\cdots,\bar\omega_{j,l}^n\}$ is the local frame of $(0, 1)$-forms constructed as above on $V^l$. Notice that $T_j u^t\in\Dom(Q^t_{\Omega_j})$ because the coefficients of $\bar\omega^J_{j, l}$ in the above expression is $0$ near $\partial\Omega_j$ if $n\in J$.
 
 \begin{thm}\label{thm5}
 	Let $\omz$ and $\omz_j$ be smooth bounded domains in $\C^n$ with normalized defining functions $r$ and $r_j$ respectively. Let $1\le q\le n-1$ and $k\in\N$. Then there exist constants $\delta$ and $C_k>0$ independent of $t$ and $j$ such that  
 	\begin{align}\label{3t1}
	\lambda^{t,q}_k(\omz_j)\le\lambda^{t,q}_k(\omz) + \dfrac{C_k\delta_j}{t^{4(2^{n+3}-1)/3}},
	\end{align}
	provided $\delta_j=\|r-r_j\|_{C^2}<\delta$. Furthermore, if $\Omega$ is pseudoconvex, then
	\begin{align}\label{3t}
		\lambda^{t,q}_k(\omz_j)\le\lambda^{t,q}_k(\omz) + \dfrac{C_k\delta_j}{t^{2^{n+3}-1}}.
	\end{align}
 \end{thm}
 \begin{proof} We provide only the proof for \eqref{3t}. The proof of \eqref{3t1} follows exactly the same lines. Let $L_{k}=\{u^t_{ 1},\cdots,u^t_{ k}\}$, where $\square^t_\omz u^t_{ h}=\lambda^t_k(\omz) u^t_{ h}$ and $\langle u^t_{ h},u^t_{ l}\rangle_{\omz}=\delta_{hl}$ for $1\leq h,l\leq k$. From elliptic theory (see Lemma~\ref{infty} above), we know that all the eigenforms $u^t_l$ are smooth up to the boundary.  We first estimate $\left|\left\langle T_ju^t_h,T_ju^t_l\right\rangle_{\omz_j}-\langle u^t_{ h},u^t_{ l}\rangle_{\omz}\right|$.  Note that $\Omega^-_{2\delta_j}\subset \Omega\cap\Omega_j$. We have
 
 \begin{equation}\label{s1}
 \begin{aligned}
 	\big|\langle T_ju^t_h,T_ju^t_l\rangle_{\omz_j}-\langle u^t_{ h},u^t_{ l}\rangle_{\omz}\big|
 	&\le\big|\langle T_j u^t_{h}, T_j u^t_{l }\rangle_{\omz^-_{2\delta_j}}-\langle u^t_{ h},u^t_{ l}\rangle_{\omz^-_{2\delta_j}}\big|\\
 	&+\big|\langle T_j u^t_{h}, T_j u^t_{l }\rangle_{\Omega_j\setminus\omz^-_{2\delta_j}}\big|+\big|\langle u^t_{ h},u^t_{ l}\rangle_{\Omega\setminus\omz^-_{2\delta_j}}\big|\\
 	&\le \big|\langle T_j u^t_{h}-u^t_h, T_j u^t_{l }\rangle_{\omz^-_{2\delta_j}}\big|+\big|\langle u^t_{ h}, T_ju^t_{l}-u^t_{ l}\rangle_{\omz^-_{2\delta_j}}\big| \\
 	&+\big|\langle T_j u^t_{h}, T_j u^t_{l }\rangle_{\Omega_j\setminus\omz^-_{2\delta_j}}\big|+\big|\langle u^t_{ h},u^t_{ l}\rangle_{\Omega\setminus\omz^-_{2\delta_j}}\big|.\\	
 \end{aligned}
 \end{equation}	
Since the corresponding coefficients of $u^t_h$ and  $T_ju^t_h$ are the same on $\Omega^-_{2\delta_j}$, we have
 \begin{equation}\label{s2}
 \|T_j u^t_h-u^t_h\|_{\Omega^-_{2\delta_j}}\le 	\sum_{l=1}^{m}\mathop{{\sum}'}\limits_{|J|=q}\big\|\psi^l u^t_{h, J} \big\|_{L^2} \|\omega_{j, l}^J-\omega_l^J\|_{C^0}\le C \delta_j.
 \end{equation}
 Note that $|\Omega_j\setminus\Omega^-_{2\delta_j}|\le C\delta_j$.  We have
 \begin{equation}\label{s3}
 	\begin{aligned}
 \big|\langle T_j u^t_{h}, T_j u^t_{l }\rangle_{\Omega_j\setminus\omz^-_{2\delta_j}}\big|&\le \|T_j u^t_h\|_{C^0} \|T_j u^t_l\|_{C^0} \big|\Omega_j\setminus \Omega^-_{2\delta_j}\big| \\
 &\le \dfrac{C_k\delta_j}{t^{2^{n+2}-1}}.
 \end{aligned}
 \end{equation}
 Here in the last inequality we have used (the proof of) Lemma~\ref{infty} and the Sobolev embedding theorem and the fact that the
 extension map  $\ec \colon W^s(\Omega)\to W^s(\C^n)$ is bounded. The other terms in \eqref{s1} are estimated similarly. We then obtain
\begin{equation}
	\left|\left\langle T_ju^t_h,T_ju^t_l\right\rangle_{\omz_j}-\langle u^t_{ h},u^t_{ l}\rangle_{\omz}\right|\le \dfrac{C_k\delta_j}{t^{2^{n+2}-1}}.
	\end{equation}
As a consequence,  $\{T_ju^t_{1},\cdots,T_ju^t_{k}\}$ is linearly independent when $\delta_0$ is sufficiently small. 

The expression $\big|Q^t_{\omz_j}(T_ju^t_h,T_ju^t_l)-Q^t_\omz(u^t_{ h},u^t_{ l})\big|$ can be estimated in the same way. The main difference is that in this case, we will need to estimate the first derivatives of $u^t_J$ and the second derivatives of $r$ and $r_j$. For example, 
we have
 \begin{align}
 	\|\dbar (T_j u^t_h-u^t_h)\|_{\Omega^-_{2\delta_j}}&\le 	\sum_{l=1}^{m}\mathop{{\sum}'}\limits_{|J|=q}\Big(\big\|\dbar (\psi^l u^t_{h, J}) \big\| \|\omega_{j, l}^J-\omega_l^J\|_{C^0}+\big\|\psi^l u^t_{h, J}\big\| \|\dbar(\omega_{j, l}^J-\omega_l^J)\|_{C^0}\Big)\notag\\
 	&\le C\|u^t_{h, J}\|_1\|r_j-r\|_{C^2}\le \dfrac{C_k \delta_j}{t^{(2^{n+3}-1)/2}}. \label{s4}
 \end{align}
Therefore we have 
\begin{equation}\label{s5}
	\left|Q^t_{\omz_j}(T_ju^t_h,T_ju^t_l)-Q^t_\omz(u^t_{ h},u^t_{ l})\right|\\ 
	\le \dfrac{C_k \delta_j}{t^{2^{n+3}-1}}.
\end{equation}
Applying Lemma \ref{first}, we then have
	\begin{align*}
	\lambda^{t,q}_k(\omz_j)\le\lambda^{t,q}_k(\omz) + \dfrac{C_k\delta_j}{t^{2^{n+3}-1}}.
	\end{align*}	
\end{proof}

We remark that in the above theorem, we need only $\Omega$ to be $C^{n+3}$-smooth.  To complete the proof of Theorem~\ref{thmt}, it remains to establish the estimate in the opposite direction:
\begin{equation}\label{s6}
	\lambda^{t,q}_k(\omz)\le\lambda^{t,q}_k(\omz_j) + \dfrac{C_k \delta_j}{t^{2^{n+3}-1}}.
\end{equation}	
The proof of \eqref{s6} is similar to that of \eqref{3t}. In this case, the transition operator is from $\Dom(Q^t_{\Omega_j})$ to $\Dom(Q^t_\Omega)$ and one need to make sure that the constants in the proofs are independent of $j$. Note that the constants in Lemma~\ref{ttt} and Lemma~\ref{infty} depend only on the diameter of $\omz$ and the $C^\infty$-norm of a defining function of $\omz$. (In fact, only derivatives up to $(n+3)^{\text{th}}$-order of the defining function have been used in the proofs.) The assumption that the $C^\infty$-norm of the defining function $r_j$ is uniformly bounded guarantees that the constants in the proofs of Lemma~\ref{ttt}, Lemma~\ref{infty}, and Theorem~\ref{thm5}  are indeed independent of $j$ when the roles of $\Omega_j$ and $\Omega$ are reversed (see the proof of Theorem~5.6 in \cite{FuZhu19} for related arguments). We leave the details to the interested reader. 

%The results in this paper can be used to study spectral stability of the $\dbar$-Neumann Laplacian. We exhibit such an application as follows. 

%\begin{cor}\label{thm3}
%	Let $\omz$ be a smooth bounded pseudoconvex domain with finite type boundary in $\C^n$ . Let $1\le q\le n-1$. Then there exist constants $1/2\ge\alpha>0$, $\delta_0>0$ and $C>0$ such that for any smooth pseudoconvex domain $\omz_j$,
%	\begin{align}\label{t3}
%		\lambda_k(\omz_j)\le\lambda_k(\omz) + Ck\tilde\delta_j^{1/(3\cdot2^{n/2+1}-2)}(\lambda_k(\omz))^{\frac{n+2}{\alpha}}
%	\end{align}
%	provided $\tilde\delta_j<\delta_0$.
%\end{cor}

%\begin{proof}
%	\begin{align*}
%		\lambda_k(\omz_2)-\lambda_k(\omz_1)&=\lambda_k(\omz_2)-\lambda^t_k(\omz_2)+\lambda^t_k(\omz_2)-\lambda^t_k(\omz_1)+\lambda^t_k(\omz_1)-\lambda_k(\omz_1)\\
%		&\le \dfrac{Ck\delta}{t^{3(2^{n/2+1}-1)}}(\lambda_k^{t}(\omz_1))^{n+2}+tC(\lambda_k(\omz_1))^{\frac{n+2}{\alpha}}
%	\end{align*}
%	Let $t=\delta^{(3\cdot2^{n/2+1}-2)^{-1}}$, thus
%	\begin{align*}
%		\lambda_k(\omz_2)-\lambda_k(\omz_1)\le Ck\delta^{1/(3\cdot2^{n/2+1}-2)}(\lambda_k(\omz_1))^{\frac{n+2}{\alpha}}.
%	\end{align*}
%\end{proof}
%\bibliography{biblio}
%\providecommand{\bysame}{\leavevmode\hbox to3em{\hrulefill}\thinspace}
%\providecommand{\MR}{\relax\ifhmode\unskip\space\fi MR }
%\providecommand{\MRhref}[2]{%
%  \href{http://www.ams.org/mathscinet-getitem?mr=#1}{#2}
%}
%\providecommand{\href}[2]{#2}

%\bibliographystyle{amsplain}

\enddocument

\end